\documentclass[12pt]{amsart}
\usepackage{amssymb,latexsym}
\usepackage{graphicx}
\newtheorem{theorem}{Theorem}

\newtheorem{definition}{Definition}

\newtheorem{questions}{Questions}
\newtheorem{example}{Example}
\newtheorem{problem}{Problem}
\newtheorem{conjecture}{Conjecture}

\newtheorem{computer search}{Computer Search}
\begin{document}

\title{The Combinatorics of Occam's Razor}
\author{William Ralph}

\address{Mathematics Department\\
Brock University\\
St. Catharines, Ontario\\
Canada L2S 3A1\\
\textnormal{Email: bralph@brocku.ca}\\
\textnormal{Phone: (905) 688-5550 x3804}}

\date{April 20, 2015}
\begin{abstract}

Occam's Razor tells us to pick the simplest model that fits our observations. In order to make sense of his process mathematically, we interpret it in the context of posets of functions. Our approach leads to some unusual new combinatorial problems concerning functions between finite sets. The same ideas are used to define a nicely behaved and apparently unknown analogue of the rank of a group.   We also make a construction that associates with each group an infinite sequence of numbers called its fusion sequence. The first term in this sequence is determined by the rank of the group and we provide examples of subsequent terms that suggest a subtle relationship between these numbers and the structure of the group.

\end{abstract}



\maketitle

\section{Introduction}\label{S:intro}

Given a choice of competing theories, Occam's razor is the principle that directs us to pick the simplest one as the most likely to be correct. This widely held rule of thumb is named after William of Occam (1285-1349), an English philosopher and logician who wrote that "plurality must never be posited without necessity" ~\cite{KK}. Occam asks us to look through a family of possible models and pick the simplest one consistent with some observed data. Consider a mathematical form of Occam's process in which the family of models is a set $A$ of functions on a set $X$, the simplicity of a model is relative to some partial order on $A$ and the data we are given are the values of some unknown function $f \in A$ on $S \subseteq X$. The key issue is how to decide if and when this data determines a unique function $f \in A$.  We make this decision as follows.

\begin{definition} \label{D:1.01} Let $(A, \leq )$ be a partially ordered set of functions on a set $X$. For $S \subseteq X$ and $ f \in A$ , we say $f$ is Occam on $S$ or $S$ is Occam for $f$  if $f$ is the least element of $\{\, g \in A  \mid g|_S = f|_S \,\}$.
\end{definition}

In the case where the partial order is equality, we are familiar with many examples where functions are completely determined by their values on particular sets. Entire functions are Occam on any open set.  Linear transformations are Occam on any basis.  Polynomials of degree at most $3$ are Occam on any four element set. Families of functions are often determined by their values on finite sets, so it will be useful to know the size of the smallest set that determines a function uniquely.    

\begin{definition} \label{D:1.02} Let $(A, \leq )$ be a partially ordered set of functions with domain $X$ and let $f \in A$.   The radius of $f$ is defined by  $R(f)=\min\{|S| \mid \text{$S \subseteq X$ and S is Occam for f } \}$ if this number is finite and $ \infty $ otherwise.  Note that if $f$ is the smallest element of $A$, then the empty set is Occam for $f$ so $R(f)= 0$. 
\end{definition}

Here is a simple example that we will shortly generalize in which the radius of each function is equal to $2$.  

\begin{example} \label{E:1.01} Write a function $f \colon \{1,2,3\} \rightarrow \{a,b\}$ as a $3$ letter word in the letters from the alphabet $\{a,b\}$ and consider the set of functions represented by  $A = \{\underline a\underline ba, \underline b\underline ab,\underline aa\underline b, \underline bb\underline a, a\underline a\underline a, b\underline b\underline b \}$. Then in $(A, = )$ each of these functions is Occam on the subset of $ \{1,2,3\}$ indicated by the two underscores.
\end{example}

In Example~\ref{E:1.01}, is it possible to make $A$ larger and still have the radius of each function equal to $2$?  This easily stated question suggests the following unusual family of difficult combinatorial problems.

\begin{problem}\label{D:7.01} For fixed m,n and r, what is the maximum possible value of $|A|$ where $|X|=m$, $|Y|=n$ and $(A, = )$ is a set of functions from $X$ to $Y$ with $R(f) \leq r$ for all $f \in A$. Let $Occ(m,n,r)$ denote this maximum value. 
\end{problem}

It is straightforward to show that $Occ(1,n,1) = n$, $Occ(m,n,1)\geq m(n-1)$ and $Occ(m,n,m) =n^m$. Here is an inequality for the general case.

\begin{theorem}\label{T:1.01} $Occ(m,n,r)$ is less than or equal to the largest number $p$ satisfying the following three conditions for some choice of the nonnegative integers $x_i$:
\begin{equation}\label{Eq:1.01} 
p=\sum_{i=1}^{n^{m-r}}ix_i \,\,\,\,\,\,\,\,\,\,\,\,\,\,\,\,\,\, \sum_{i=1}^{n^{m-r}}x_i \leq n^r  \,\,\,\,\,\,\,\,\,\,\,\,\,\,\,\,\,\, p\leq {m \choose r}x_1
\end{equation}

\end{theorem}
\begin{proof} 

Let $A$ be any set of functions from an $m$ element set $X$ to an $n$ element set $Y$.  Given an $r$ element subset  $S \subseteq X$, define an equivalence relation on $A$ by $f \sim_{S} g$ if $f_{|S} = g_{|S}$. Suppose that the distinct nonempty elements of the partition of $A$ induced by $\sim_{S}$ are $P_1, P_2, \dots, P_k$ where $k \leq n^r$ since there are at most $n^r$ distinct functions in the restriction of functions in $A$ to $S$  .  There are at most $n^{m-r}$ functions that agree with any particular $f$ on $S$ so $|P_i| \leq n^{m-r}$. If $x_i$ is the number of times that the number $i$ occurs in the list  $|P_1|, |P_2|, \dots, |P_k|$, then $|A| =\sum_{i=1}^{k}|P_i| =  \sum_{i=1}^{n^{m-r}}ix_i $. Since $k \leq n^r$, we must also have that $\sum_{i=1}^{n^{m-r}}x_i  \leq n^r$.  Now suppose that  $S_1, S_2, \dots, S_{m \choose r} $ are all of the $r$ elements subsets of $X$  and let $x_{1,j}$ denote the number of singleton sets in the partition induced by the equivalence relation $\sim_{S_j}$.  By our assumption that every element in $A$ has radius less than or equal to $r$, we must have $\sum_{j=1}^{m \choose r}x_{1,j} \geq |A|$. It follows that for at least one value $j_0$ of $j$ we must have $x_{1,j_0} \geq |A|/{m \choose r}$ and the result follows.
\end{proof}

Theorem~\ref{T:1.01} can be used in conjunction with a computer to find upper bounds for $Occ(m,n,r)$ for small values of $m$, $n$ and $r$.

\begin{computer search}\label{CS:1.01} If we fix $m=3$ and $r=2$ and consider the values $n=2,3,4,5$ in Equation~\ref{Eq:1.01}, then the corresponding maximum possible values of $p$ are $6,15,31,53$ respectively.  
\end{computer search}

It follows that $Occ(3,3,2) \leq 6$, $Occ(3,3,2) \leq 15$ and $Occ(3,4,2) \leq 31$ but it remains to determine if these upper bounds can actually be attained. Here is what is known about the radius $2$ case.

\begin{theorem}\label{T:7.07} 
\noindent
\begin{enumerate}
\item $Occ(3,2,2)=6$
\item $Occ(3,3,2)=15$
\item  $28 \leq Occ(3,4,2) \leq 31$
\item  $4 \binom{n}{2}+n \leq Occ(3,n,2) $
\item  $2m\leq Occ(m,2,2)$ 
\end{enumerate}
\end{theorem}

\begin{proof} Parts $(1)$ to $(4)$ of this theorem concern $Occ(3,n,2)$ which we will bound by generalizing the pattern in Example~\ref{E:1.01}. There are $ \binom{n}{2}$ different pairs of letters. For each pair $\{u,v\}$ taken from the n letters in the alphabet, we add the four functions $\underline u\underline vu$,$\underline v\underline uv$, $\underline uu\underline v$, $\underline vv\underline u$ for a total of $4 \binom{n}{2}$	functions. We then add the $n$ constant functions of the form $u\underline u\underline u$ for a total of  $4 \binom{n}{2}+n$ functions which are all Occam on the $2$ underscored positions.  It follows that  $4 \binom{n}{2}+n \leq Occ(3,n,2) $ and parts $(1)$ to $(3)$ now follow from the results of Computer Search~\ref{CS:1.01}.

For Part $(5)$, consider the special case where $m=5$.  The $5$ functions $\underline a\underline aaaa$,$a\underline b\underline aaa$, $ab\underline b\underline aa$, $abb\underline b\underline a$ and $\underline abbb\underline b$ together with the $5$ functions $\underline b\underline bbbb$, $b\underline a\underline bbb$, $ba\underline a\underline bb$ , $baa\underline a\underline b$ and $\underline baaa\underline a$ are all Occam on the $2$ underscored positions.This pattern generalizes to give the result.

\end{proof}

\section{The Radius of a Subgroup}\label{S:radius}
In this section, we develop $Occ(G)$ for a group $G$ which is analogous to the rank of $G$ but has the nice property that if $H$ is a subgroup of $G$ then $Occ(H) \leq Occ(G)$. In order to apply the ideas of the last section, we will associate with any group $G$ the set of functions $A_G$ in the following way.
 \begin{definition} \label{D:2.01} Let $G$ be any group and $H$ be any subgroup of $G$. Define $\chi_H : G \rightarrow \{0,1\}$ to be the characteristic function of $H$ that takes the value $1$ at each element of $H$ and the value $0$ at elements of $G$ not in $H$.  Let  $A_G$ denote the set of all characteristic functions of all subgroups of $G$. 
\end{definition}
The elements of $A_{G}$ for the dihedral group  $D_4$ of order $8$ are listed on the left hand side Table~\ref{Tab:2.01}. We now use Definition~\ref{D:1.02} and define the radius of a subgroup to be the radius of its characteristic function.
\begin{definition} \label{D:2.01} Let $G$ be any group and $H$ be any subgroup of $G$. Define the radius of $H$ in $G$ or $R_G(H)$  to be the radius of $\chi_H$ in $(A_G, =)$.
\end{definition}

Table~\ref{Tab:2.01} shows the radii of the subgroups of $D_4$.  These numbers capture information about how the subgroup sits inside the group.  

\begin{table}[!ht]{The Radius of Subgroups of $D_4$}
\begin{center}
\begin{tabular}{|c|c|}

 \hline
\textit{$\chi_H$} & $R_{D_4}(H)$     \\
\hline
11111111 & $2$   \\
\hline
10000000 & $5$  \\
\hline
10000100 & $2$    \\
\hline
10000100 & $2$     \\
\hline
10000010 & $2$     \\
\hline
10000001 & $2$   \\
\hline
10100000 & $4$   \\
\hline
11110000 & $2$    \\
\hline
10101010 & $3$    \\
\hline
10100101 & $3$    \\
\hline

\end{tabular}
\caption{The radii of the subgroups of $D_4$ found by a computer program.}\label{Tab:2.01}
\end{center}
\end{table}

\begin{theorem}\label{T:2.02} If $G$ is any group, then $R_G(G)$ is the rank of $G$.
\end{theorem}
\begin{proof} 
Suppose that rank of $G$ is $n$ and that $S=\{g_1, g_2, \dots , g_n \}$ generates $G$. $S$ must be Occam for $\chi_G$ because if $\chi_H =\chi_G$ on $S$ then $H$ contains $S$ which implies $H = G$ and therefore $\chi_H =\chi_G$. It follows that $R(\chi_G) \leq n$. Now suppose by way of contradiction that $R(\chi_G) =m < n$ and that $T=\{g_1, g_2, \dots , g_m \}$  is Occam for $\chi_G$.  If $H = <g_1, g_2, \dots , g_m>$ then $\chi_H =\chi_G$ on $T$. But $H$ must be a proper subgroup of $G$ which means  $\chi_H  \neq \chi_G$ and $T$ is not Occam for $\chi_G$.
\end{proof}

We will be particularly interested in the radius of the identity subgroup of a group $G$ which we will write as $Occ(G)$.

\begin{definition} \label{D:2.02} If $G$ is any group, define $Occ(G) = R_G(\{e\})$.
\end{definition}

Values of $Occ(G)$ are shown for various groups in Table~\ref{Tab:2.02} and we see that these small groups are classified by their order, rank and $Occ(G)$. 

\begin{table}[!ht]{The Radius of the Identity for Various Groups}
\begin{center}
\begin{tabular}{|c|c|c|c|}
 \hline
\textit{Group G} & Order &  Rank & $Occ(G)=  R_G(\{e\})$     \\
\hline
$Z_2$ & $2$ & $1$  & $1$     \\
\hline
$Z_3$ & $3$ & $1$  & $1$     \\
\hline
$Z_4$ & $4$ & $1$  & $1$     \\
\hline
$Z_2 \times Z_2$ & $4$ & $2$  & $3$     \\
\hline
$Z_5$ & $5$ & $1$  & $1$     \\
\hline
$Z_6$ & $6$ & $1$  & $1$     \\
\hline
$S_3$ & $6$ & $2$  & $4$     \\
\hline
$Z_7$ & $7$ & $1$  & $1$     \\
\hline
$Z_8$ & $8$ & $1$  & $1$     \\
\hline
Q & $8$ & $2$  & $1$     \\
\hline
$Z_2 \times Z_4$ & $8$ & $2$  & $3$     \\
\hline
$D_4$ & $8$ & $2$  & $5$     \\
\hline
Dicylic & $12$ & $2$  & $2$     \\
\hline
$Z_2 \times Z_6$ & $12$ & $2$  & $4$     \\
\hline
$A_4$ & $12$ & $2$  & $7$     \\
\hline
$D_6$ & $12$ & $2$  & $8$     \\
\hline

\end{tabular}
\caption{ The values of $Occ(G)$ found by a computer program.}\label{Tab:2.02}
\end{center}
\end{table}

Unlike the rank, $Occ$ has the nice property that if $H$ is a subgroup of $G$ then $Occ(H) \leq Occ(G)$ as we now show.
\begin{theorem} \label{T:2.02} If $H$ is a subgroup of $G$ then $Occ(H) \leq Occ(G)$.
\end{theorem}
\begin{proof} It is enough to show that if $U$ is Occam for $\chi_{\{e\}}$ in $(A_G,=)$ then $U\cap H$ is Occam for $\chi_{\{e\}}$ in $(A_H,=)$; the relationship between the radii will then follow because $|U\cap H| \leq |U|$. Now $U$ is Occam for $\chi_{\{e\}}$ in $(A_G,=)$ means that if $K$ is any subgroup of $G$ with $U \cap K = U \cap \{e\}$ then $K = \{e\}$.  To see that $U\cap H$ is Occam for  $\chi_{\{e\}}$ in $(A_H,=)$, suppose that $L$ is any subgroup of $H$ and that  $(U \cap H) \cap L = (U \cap H) \cap \{e\}$.  This equality simplifies to $ U \cap L = U \cap \{e\}$ which implies $L = \{e\}$ since $U$ is Occam for $\chi_{\{e\}}$ in $(A_G,=)$. 
\end{proof}
Using Theorem~\ref{T:2.02} and Table~\ref{Tab:2.02} we can verify the well-known relationships that $Z_2 \times Z_2$ cannot be a subgroup of the Quaterion group of order $8$ or the Dicylic group of order $12$ and also that $S_3$ is not a subgroup of the Dicylic group.

\begin{questions} \label{Q:2.01} What is the smallest example of two nonisomorphic groups with the same order, rank and Occ?  How do we characterize the groups $G$ with $Occ(G)=1$ which include the cyclic groups and the Quaternion group of order $8$ ?
\end{questions}

\section{The Fusion Sequence of a Group}\label{S:fusion}
We will make a construction that allows us to associate with any poset of functions $(A,\leq)$ an infinite sequence of positive integers called its fusion sequence.  The first step is to identify some special subsets of $A$.

\begin{definition} \label{D:3.01} Let $(A, \leq )$ be a partially ordered set of functions on a set $X$ and let $S$ be a subset of $A$.  Define the fusion set of $S$ to be the set $F_S$ of all functions $f \in A$ that are Occam on $S$.
\end{definition}

Give a function $f$ in some poset of functions, we will want to know the size of the smallest fusion set containing $f$. We call this number the fusion number of $f$.

\begin{definition} \label{D:3.02} Let $(A, \leq )$ be a partially ordered set of functions on a set $X$ and let $f \in A$.  Define the fusion number of $f$ to be $\min\{|F_S|  \mid f \in F_S\} $  if finite and infinity otherwise.
\end{definition}

The next construction starts with a poset of functions $(A, \leq )$ and uses its fusion sets to build a new poset of functions on the set $A$.

\begin{definition} \label{D:3.03} Let $(A, \leq )$ be a partially ordered set of functions on a set $X$. Define an infinite sequence of partially ordered sets $(A, \leq )_1, (A, \leq )_2, \dots$, called the ascendents of $(A, \leq )$, as follows.  The first ascendent of $(A, \leq )$ is $(A, \leq  )_1=(\{\chi_{F_S} | S \subseteq X \},\prec)$ where $\prec$ orders the characteristic functions by subset inclusion and the domain of each of the functions $\chi_{F_S} $ is the set $A$. Now define $(A, \leq  )_2=(\{\chi_{F_S} | S \subseteq X \},\prec)_1$ and continue inductively.
\end{definition}

We note that $\chi_A$ is the maximal element of $(A, \leq )_1$ which implies that every ascendent of $(A, \leq )$ has a maximum element.  The radii of these maximal elements are the terms in the fusion sequence of $(A, \leq )$ that we now define.

\begin{definition} \label{D:3.04} Let $(A, \leq )$ be a partially ordered set of functions on a set $X$. Each ascendent of $(A, \leq )$ has a maximal element so we let $F_n$ be the fusion number of the maximal element of the nth sequent of $(A, \leq )$ .  The sequence $F_1,F_2, \dots $ will be called the fusion number sequence of $(A, \leq )$. If $(A, \leq )$ happens to have a maximal element, then we denote the fusion number of that element by $F_0$.
\end{definition}

Within this general context of fusion sequences, we now focus on the following poset of functions associated with a group $G$.

\begin{definition} \label{D:3.05} Let $G$ be any group and $H$ be any subgroup of $G$. Define $\chi_H : G \rightarrow \{0,1\}$ to be the characteristic function of $H$ that takes the value $1$ at each element of $H$ and the value $0$ at elements of $G$ not in $H$.  Let  $(A_G, \prec)$ denote the set of all characteristic functions of all subgroups of $G$ ordered by inclusion.  In other words, $\chi_H \prec  \chi_K$ if $H \subseteq K$.
\end{definition}

If $G$ is any group then we can give a precise description of the fusion sets.

\begin{theorem} \label{T:3.01} Let $G$ be any group. If $S$ is any subset of $G$ and $F_S$ is the fusion set of $S$ with respect to $(A_G, \prec)$, then $F_S = \{ \chi_{<R>} |\mid R \subseteq S\}$. In other words, $F_S$ consists of the characteristic functions of all subgroups generated by all possible subsets of $S$ with the convention that the empty set generates the identity subgroup.
\end{theorem}
\begin{proof} 

To see that$\{ \chi_{<R>} |\mid R \subseteq S\}  \subseteq  F_S$, 
let $R\subseteq S$ and suppose that $\chi_H$   agrees with $\chi_{<R>}$ on $S$.  Then $\chi_H$   agrees with $\chi_{<R>}$ on $R$ and therefore $R \subseteq H$ and $ <R> \subseteq H$ so $\chi_{<R>}$ is Occam on $S$.
For the reverse inclusion, suppose $\chi_H \subseteq F_S$.  Let $R=H \cap S$ which implies $<R> \subseteq  H $ and $<R> \cap S =  H \cap S$.  It follows that $\chi_{<R>} \prec \chi_H$ and that $\chi_{<R>}$ agrees with  $\chi_H$ on $S$. But by assumption $S$ is Occam for $\chi_H$ so we must have $\chi_{<R>} = \chi_H$.
\end{proof}

Table~\ref{Tab:3.01} shows the first few terms in the fusion sequences of various groups. Here is a sample calculation for the group $G = Z_4$.

\begin{example}\label{E:3.23} Consider the calculation of the fusion sequence of $(A_G, \prec)$ in the case where $G = Z_4 = \{0,1,2,3\}$. In this case, 

$A_{Z_4}=\{(1,0,0,0),(1,0,1,0),(1,1,1,1)\}$ and the fusion number of the maximal element $(1,1,1,1)$ is $F_0 =2$.  

$(A_G, \prec)_1 =$$(\{(1,0,0), (1,1,0), (1,0,1), (1,1,1)\},\prec)$ and the fusion number of the maximal element $(1,1,1)$ is $F_1 =4$. 

$(A_G, \prec)_2 =(\{(1,0,0,0), (1,1,0,0), (1,0,1,0) , (1,1,1,1)\},\prec)$ and the fusion number of the maximal element $(1,1,1,1)$ is $F_2 =2$.
\end{example}

\begin{table}[!htb]\label{Tab:3.01}
\begin{center}
\begin{tabular}{|c|c|c|c|c|}
 \hline
Group & $F_0$  & $F_1$  & $F_2$  & $F_3$  \\
\hline
{$Z_p$ } \textit{}p a prime  & $2$  & $2$  & $2$  & $2$  \\
\hline
{$Z_4$}  & $2$  & $4$  & $2$  & $4$  \\
\hline
{$Z_2 \times Z_2 $}  & $4$  & $8$  & $2$  & ?   \\
\hline
{$Z_6$}  & $2$  & $4$  & $4$  & $8$  \\
\hline
{$S_3$}  & $4$  & $16$  & $2$  & ?   \\
\hline
{$Z_8$}  & $2$  & $8$  & $2$  & $8$  \\
\hline
{$D_4$}  & $4$  & $64$  & ?  & ?  \\
\hline
{$Q$}  & $4$  & $16$  & $2$  & ?  \\
\hline
{$Z_9$}  & $2$  & $4$  & $2$  & $4$  \\
\hline
{$Z_3 \times Z_3 $}  & $4$  & $16$  & $2$  & $?$   \\
\hline
{$Z_{10}$}  & $2$  & $4$  & $4$  & $8$   \\
\hline
{$Z_{12}$}  & $2$  & $8$  & $4$  & ?   \\
\hline
{$Z_{14}$}  & $2$  & $4$  &  $4$  &  $8$  \\
\hline
{$Z_{16}$}  & $2$  & $16$  & $2$  & ?   \\
\hline

\end{tabular}
\caption{The first four fusion numbers of $(A_G, \prec)$ found with a computer program.}
\end{center}
\end{table}

The values of $F_0$ in Table~\ref{Tab:3.01} suggest the following theorem.

\begin{theorem} \label{T:3.02} If $G$ is any group with finite rank $n$ then $F_0 = 2^n$. 
\end{theorem}
\begin{proof} 
The maximal element of $(A_G, \prec)$ is $\chi_G$ and  $F_0= \displaystyle \min\{|F_S|  \mid \chi_G \in F_S\}$ by Definition~\ref{D:3.03}. By Theorem~\ref{T:3.01}, $F_S = \{ \chi_{<R>} |\mid R \subseteq S\}$ which means $F_0=  \min  \{|F_S|  \mid \text{S generates G }\}$.  Suppose $T = \{g_1, g_2, \dots , g_n \}$  generates $G$ where $n$ is the rank of $G$.  Then every subset of $T$ generates a distinct subgroup of $G$ whose characteristic function is Occam on $T$ and therefore $F_0 \leq |F_T| = \{ \chi_{<R>} |\mid R \subseteq T\}= 2^n$. It remains to show that if $S$ is any generating set for $G$ then $|F_S| \geq 2^n$. For each $i$, we can find a finite number of elements of $S$ whose product is $g_i \in T$.  Together, all of these elements of $S$ used to generate the elements of $T$ form a finite generating set of $G$ which must contain a set $U$ that generates $G$ and is minimal in the sense that no proper subset of $U$ generates $G$.  If $R_1$ and $R_2$ are two different subsets of $U$ then we cannot have $<R_1>$ = $<R_2>$ because then at least one of the elements of $U$ could be expressed in terms of other elements of $U$ contradicting $U$ being a minimal generating set.  Consequently, the set $\{ <R> \mid R \subseteq U\}$ must consist of precisely $2^{|U|}$ different subgroups where as usual we take  $ < \emptyset > = \{e\}$ . Therefore $|F_S| = |\{<R>\mid R\subseteq S \}|$ $\geq$   $ |\{<R>\mid R\subseteq U \}|$ $=$  $2^{|U|}$  $\geq$   $2^n$, since $|U| \geq n= rank(G)$ by the definition of rank.
\end{proof}

It is well known that the rank of a group is not effectively computable in general so it follows from Theorem~\ref{T:3.02} that the same must be true for $F_0$. On almost no evidence we make the following conjecture:

\begin{conjecture} \label{C:3.01} If $G$ is any finite group then the fusion sequence of $G$ is periodic. 
\end{conjecture}

We have used maximal elements and subset inclusion to define the fusion sequence but there are other possible approaches that we will consider elsewhere.

\end{document}